\documentclass[12pt,a4paper]{article}
\usepackage{a4wide,amsfonts,amsmath,latexsym,amssymb,euscript,eufrak,graphicx,units,mathrsfs}
\usepackage[utf8]{inputenc}
\usepackage{amsmath}
\usepackage{amsfonts}
\usepackage{amssymb}
\usepackage{amsthm}
\usepackage{floatrow}
\usepackage{blindtext}

\usepackage[usenames,dvipsnames]{color}

\numberwithin{equation}{section}

\newcommand{\R}{\mathbb{R}}

\newcommand{\N}{{\mathbb N}}

\newcommand{\ff}{{i=1,\ldots,p}}

\newtheorem{theorem}{Theorem}[section]

\newtheorem{lemma}{Lemma}[section]

\begin{document}

\renewcommand{\thefootnote}{\arabic{footnote}}

\begin{center}
{\Large \textbf{Statistical analysis of the  non-ergodic
fractional Ornstein-Uhlenbeck process with periodic mean}} \\[0pt]
~\\[0pt]
Rachid Belfadli\footnote{Research Group of Geometry, Stochastic
Analysis and Applications, Department of Mathematics, Faculty of
Sciences and Techniques, Cadi Ayyad University, Marrakech, Morocco.
E-mail:
 \texttt{r.belfadli@uca.ma}}\  Khalifa Es-Sebaiy\footnote{ Department of Mathematics, Faculty of
Science, Kuwait University, Kuwait. E-mail:
 \texttt{khalifa.essebaiy@ku.edu.kw}}\
  Fatima-Ezzahra Farah\footnote{National School of Applied Sciences-Marrakech, Cadi Ayyad University, Marrakech, Moroccco. E-mail:
 \texttt{farah.fatima.ezzahra@gmail.com}}
\\[0pt]
~\\[0pt]
 Cadi Ayyad University and Kuwait University \\[0pt]
~\\[0pt]
\end{center}

\begin{abstract}
Consider a periodic, mean-reverting Ornstein-Uhlenbeck process
$X=\{X_t,t\geq0\}$ of the form $d X_{t}=\left(L(t)+\alpha
X_{t}\right) d t+ dB^H_{t}, \quad t \geq 0$, where
$L(t)=\sum_{i=1}^{p}\mu_i\phi_i (t)$ is a periodic parametric
function, and $\{B^H_t,t\geq0\}$ is a fractional Brownian motion of
Hurst parameter $\frac12\leq H<1$. In the ``ergodic'' case
$\alpha<0$, the parametric estimation of
$(\mu_1,\ldots,\mu_p,\alpha)$ based on continuous-time observation
of $X$ has been considered in Dehling et al. \cite{DFK}, and in
Dehling et al. \cite{DFW} for $H=\frac12$, and $\frac12<H<1$,
respectively. In
 this paper we consider the ``non-ergodic'' case   $\alpha>0$, and for all $\frac12\leq
 H<1$. We analyze the strong consistency and the asymptotic distribution for  the  estimator of
 $(\mu_1,\ldots,\mu_p,\alpha)$ when the whole trajectory of $X$ is observed.
\end{abstract}


\noindent {\bf 2010 AMS Classification Numbers:}  60G15; 60G22;
62F12; 62M09; 62M86.

\section{Introduction}\label{sec1}
While the statistical inference of It\^o type diffusions has a long
history  (see e.g. Basawa and  Scott \cite{BS}, Kutoyants
\cite{kutoyants}, Liptser and Shiryaev \cite{LS} and the references
therein), the statistical analysis for equations driven by
fractional Brownian motion (fBm) is relatively recent. The
development of stochastic calculus with respect to the fBm allowed
to study such models. Estimation of the drift parameters in
fractional-noise-driven Ornstein-Uhlenbeck processes is a problem
that is both well-motivated by practical needs and theoretically
challenging. In the finance context, a practical motivation to study
this estimation problem is to provide tools to understand volatility
modeling in finance. Indeed, any mean-reverting model in discrete or
continuous time can be taken as a model for stochastic volatility.
Recently, several researchers have been interested in studying
statistical estimation problems when the volatility exhibits
long-memory, which means that the volatility today is correlated to
past volatility values with a dependence that decays very slowly.

 In the present paper we consider a fractional Ornstein-Uhlenbeck
 (fOU)
process with periodic mean $X:=\lbrace X_t, t\geq 0\rbrace $,
defined as solution of the following linear stochastic differential
equation
\begin{equation}\label{PeriodicMean}
  X_0=0, \quad dX_t=(L(t) +\alpha X_t)dt + dB^H_t, \quad  t\geq 0,
\end{equation}
where $B^H=\lbrace B^H_t, t\geq 0\rbrace $ is a fractional Brownian
motion with Hurst index $ \frac{1}{2}\leq H<1$, $\alpha >0$ and
$L(t)=\sum_{i=1}^{p}\mu_i\phi_i (t)$ for all $t\in [0, 1]$
 with $(\mu_1,\ldots, \mu_p)\in \R^{p}$, and     the functions $\phi_i$,  $i=1,\ldots,p$, are bounded $1$-periodic $L^2([0, 1])$-orthonormal
 functions. Let $\mu:=\left(\mu_1,\ldots,\mu_p\right)$, and $\theta:=(\mu,
 \alpha)$.\\

 Let us recall some results on parameter estimation related
 to the process \eqref{PeriodicMean}:
 \begin{itemize}
 \item ``Ergodic'' case $\alpha<0:$ When $H=\frac12$, the maximum likelihood estimator (MLE) of $\theta$  has been studied by \cite{DFK} based on continuous-time observation
 of process $X$, defined by \eqref{PeriodicMean}. The authors
  proved the strong consistency and
asymptotic normality of the MLE. Moreover,   at this stage it is
worth noticing that the MLE coincides with the least squares
estimator (LSE) as well. When $\frac12<H<1$, the consistency and
asymptotic normality of the LSE of $\theta$ based on continuous-time
observation of $X$, have been studied in \cite{DFW,BEV2017}. Also,
some non-Gaussian extensions of the model \eqref{PeriodicMean} have
been considered by several authors (see e.g. \cite{NT,ST}), by
replacing the fBm in \eqref{PeriodicMean} by a Hermite process. On
the other hand, for $p=1$ and $\phi_1=1$, a large number of research
articles considered the problem of drift parameter estimation for
various fractional diffusions and in particular for the fOU process,
we refer among many others to \cite{KL,HNZ,EEV,EV,DEV,SV}.

 \item ``Non-ergodic'' case $\alpha>0:$  When
$p=1$ and $\phi_1=1$, several researchers have been interested in
studying parameter estimation problems in various fractional
Gaussian models related to \eqref{PeriodicMean}, we refer among many
others to \cite{BEO,EEO,EE,AAE,EAA}.
 \end{itemize}

However, in the ``non-ergodic'' case corresponding to $\alpha>0$, no
authors as far as we know have ever studied the drift parameter
estimation for the model \eqref{PeriodicMean} in its general form.
So, our goal in the present paper is to consider
 the MLE  and the LSE for $\theta$, respectively, when $H=\frac12$ and
 $\frac12<H<1$, based on continuous-time observation of $X$. We study the strong consistency and the
asymptotic distribution for those estimators.\\

The rest of the paper is structured as follows. In Section 2 we
analyze some properties of our model. In Section 3 we prove  the
strong consistency for  estimator $\hat{\theta}_n=(\hat{\mu}_n,
\hat{\alpha}_{n})$, defined by \eqref{estimator-theta}. Section 4 is
devoted to the joint asymptotic distribution of $\left(e^{\alpha
n}(\hat{\alpha}_n-\alpha),n^{1-H}\left(\hat{\mu}_n-\mu
\right)\right)$, as $n\rightarrow\infty$.\\

In what follows, $C$   denotes a generic positive constant (perhaps
depending on $H$ and $\alpha$, but not on anything else), which may
change from line to line.

\section{Some almost sure convergence properties}
In this section, we  study some properties of our model
\eqref{PeriodicMean}, which rely on the periodicity of the mean
function and the fact that $\alpha>0$. These properties will be
needed  in order to analyze the
asymptotic behavior of the LSE.\\

 Consider the fOU process with periodic mean that is defined as the solution
to the Langevin equation whose drift is a periodic function, given
by \eqref{PeriodicMean}.\\  We assume that the parameter
$\theta=(\mu, \alpha)$ is unknown and our aim is to estimate it by
using maximum likelihood method when $H=\frac12$, and least squares
method when $\frac12<H<1$. More precisely, we consider the following
LSE $\hat{\theta}_n$ for $\theta$, which coincides with the MLE when
$H=\frac12$:

\begin{align}
\hat{\theta}_n:= Q_n^{-1}P_n,\label{form1-theta-estimator}
\end{align}
where $$P_n:=\left(\int_0^n \phi_1(s)dX_s, \cdots, \int_0^n
\phi_p(s)dX_s, \int_0^nX_sdX_s\right)^{T} \, \,\,\,
\mbox{and}\,\,\,\, Q_n:=\begin{pmatrix}
A_{p, n} & U_n \\\\
U_n^{T} & V_n
\end{pmatrix},$$
with $$U_n=\left(\int_0^n \phi_1(s)X_sds, \cdots, \int_0^n
\phi_p(s)X_s ds\right)^{T}, \,\,  V_n= \int _0^n X^2_s ds,$$ and
$$  A_{p, n}:= \left(\int_0^n \phi_i(s)\phi_j(s)ds\right)_{1\leq i, j \leq p }= nI_p, \,\, \mbox{since}\,\,\int_0^1 \phi_i(s)\phi_j(s)ds
=  \left\{
\begin{array}{ll}
1 & \mbox{ if }i=j, \\
\\
0 & \mbox{ if }i\neq j.
\end{array}%
\right.$$ Thus \[Q_n:=\begin{pmatrix}
nI_p & U_n \\\\
U_n^{T} & V_n
\end{pmatrix}.\]
Straightforward calculation yields that for every $a\in\R^p$ and
$b\in\R$,
$$\begin{pmatrix}
I_p & a \\\\
a^{T} & b
\end{pmatrix}^{-1}=
\begin{pmatrix}
I_p+\frac{aa^T}{b-\|a\|^2} & \frac{-a}{b-\|a\|^2} \\\\
\frac{-a^T}{b-\|a\|^2} & \frac{1}{b-\|a\|^2}
\end{pmatrix}.
$$
Thus the explicit expression of the inverse matrix $Q_n^{-1}$ can be
written as:
$$Q_n^{-1}= \frac{1}{n}\begin{pmatrix}
I_p+\gamma_n\Lambda_n\Lambda_n^T & -\gamma_n\Lambda_n \\\\
-\gamma_n\Lambda_n^T& \gamma_n
\end{pmatrix},$$
where \[\Lambda_n:=(\Lambda_{n,1}, \ldots,
\Lambda_{n,p})^T:=\frac{1}{n}U_n,\quad
\gamma_n:=\left(\frac{1}{n}\int_0^n X_s^2 ds -\sum_{i=1}^{p}
\Lambda_{n,i}^2 \right)^{-1}.\] On the other hand, the explicit
strong solution of the Langevin equation (\ref{PeriodicMean}) is
given by
\begin{eqnarray*}
X_t&=&e^{\alpha t}\left(\int_{0}^{t} e^{-\alpha s} L(s) ds
+\int_{0}^{t} e^{-\alpha s} dB^H_s\right)\\
&=:&e^{\alpha t}\left(A_t+\zeta_t\right),\quad t\geq 0,
\end{eqnarray*}
where, for every $t\geq 0$,
\begin{equation}
A_t:=\int_{0}^{t} e^{-\alpha s} L(s) ds;\quad \zeta_t:=\int_{0}^{t}
e^{-\alpha s} dB^H_s=e^{-\alpha t}B^H_t+\alpha Z_t\label{zeta},
\end{equation}
 with
\begin{equation}Z_t:=\int_{0}^{t} e^{-\alpha s} B^H_sds, \quad t\geq 0.\label{exp-Z}
 \end{equation}
Therefore, the process $X$ can  be rewritten as
\begin{equation}
X_t=e^{\alpha t}A_t+\alpha e^{\alpha t}Z_t+ B^H_t, \quad t\geq 0.
\label{form_X_t}
\end{equation}
On the other hand, from Equation (\ref{PeriodicMean}), we can also
write
\begin{equation}
X_t=\tilde{L}(t)+\alpha\Sigma_t+B^H_t,\quad t\geq
0,\label{form_X_Sigma}
\end{equation}
where \begin{equation}\tilde{L}(t):=\int_0^t L(s) ds,\quad
\Sigma_t:=\int_0^t X_s ds,\quad
t\geq0.\label{def_L_Sigma}\end{equation}
 Note that  almost surely
all paths of $X$ are $\gamma$-H\"older continuous with $\gamma<H$.
Indeed,  it is well known that the trajectories of $B^H$ are
$\gamma$-H\"older continuous with $\gamma<H$ (see, for example,
\cite[page 274]{nualart-book}). Moreover, since $L(t)$ is bounded,
we have, for every $s,t\geq0$,
\[\left|A_t-A_s\right|\leq \int_{s}^{t} e^{-\alpha r} |L(r)| dr\leq C |t-s|,\]
which implies that the $A_t$ $1$-H\"older continuous. Furthermore,
we have for every $s,t\in[0,T]$,
\[|Z_t-Z_s|=\left|\int_{s}^{t} e^{-\alpha r} B^H_rdr\right|\leq \sup_{r\in[0,T]} |B^H_r| |t-s|,\]
which proves that the trajectories of $Z$ are $1$-H\"older
continuous.  Combining these facts together with \eqref{form_X_t},
the desired result is obtained.
\\
Therefore, the estimator \eqref{form1-theta-estimator} is well
defined in the following sense:
\begin{itemize}
\item For $H=\frac12$, the stochastic integrals in the expression
\eqref{form1-theta-estimator} of $\hat{\theta}_n$  are understood in
the It\^o sense. Moreover the integral, that is the pth component of
$P_n$ above, $\int_{0}^{n} X_{s} d X_{s}=\frac{1}{2}\left(X_{n}^{2}
-n\right)$, by  It\^o formula.
\item For $\frac12<H<1$, the stochastic integrals in the expression
\eqref{form1-theta-estimator} of $\hat{\theta}_n$  are understood in
the Young sense (see Appendix). In addition the integral, that is
the pth component of $P_n$ above, $\int_{0}^{n} X_{s} d
X_{s}=\frac{1}{2}X_{n}^{2}$,    using the integration by parts
formula \eqref{IBP}.
\end{itemize}
Thus, in the rest of the paper,  we will use
$\frac{1}{2}\left(X_{n}^{2} -n\right)$ instead of $\int_{0}^{n}
X_{s} d X_{s}$ when  $H=\frac12$, and $\frac{1}{2}X_{n}^{2}$ instead
of $\int_{0}^{n} X_{s} d X_{s}$ when  $\frac12<H<1$. We will also
make use of the following form of $\hat{\theta}_n$:
\begin{align}
\hat{\theta}_n:&=(\hat{\mu}_n,
\hat{\alpha}_{n}),\label{estimator-theta}
\end{align}
where
\begin{align}\label{estimator1}
\hat{\alpha}_{n}:= \frac{\gamma_n}{n}\left(\int_{0}^{n} X_{s} d
X_{s}-\sum_{k=1}^p \Lambda_{n, k} \int_{0}^{n} \phi_k(s)dX_s\right),
\end{align}
and   $\hat{\mu}_n:=\left(\hat{\mu}_{n, 1},\ldots,\hat{\mu}_{n,
p}\right)$ such that, for all $i=1,\ldots, p$,
\begin{align}\label{estimator2}
\hat{\mu}_{n, i}:&=\frac{1}{n}\left(\int_{0}^{n} \phi_{i}(s)dX_s +
\gamma_n \Lambda_{n, i}\sum_{k=1}^p \Lambda_{n, k} \int_{0}^{n}
\phi_k(s)dX_s -\gamma_n \Lambda_{n, i}  \int_{0}^{n} X_{s} d
X_{s}\right).
\end{align}

To prove our main results we will need the following lemmas.

\begin{lemma}\label{Lemma 1} Assume $\frac12\leq H< 1$. Let $\{\zeta_t,t\geq0\}$ and $\{Z_t,t\geq0\}$ be given by \eqref{zeta} and \eqref{exp-Z}. Then,
\begin{eqnarray}
\lim_{t\rightarrow\infty}\frac{B^H_t}{t^{\delta}}= 0\ \mbox{ almost
surely  for all } H<\delta<1,\label{i_lemma1}
\end{eqnarray}
\begin{eqnarray}                  \lim_{t\rightarrow\infty}Z_t= Z_{\infty}:=\int_0^{\infty}e^{-\alpha s}B^H_s ds \ \mbox{ almost
surely and in } L^2(\Omega),\label{ii_lemma1}
\end{eqnarray}
\begin{eqnarray}  \lim_{t\rightarrow\infty}e^{-2 \alpha t}
\int_{0}^{t} e^{2 \alpha s} \zeta_{s}^{2} ds =\frac{\alpha}{2}
Z_{\infty}^{2} \quad \text { almost surely},\label{iii_lemma1}
\end{eqnarray}
\begin{eqnarray}
\lim_{n\rightarrow\infty}\frac{1}{n}\int_0^n \phi_i(s)dB^H_s= 0
\quad \text { almost surely  for every } i=1,\ldots,p.\label{cv_fBm}
\end{eqnarray}
\end{lemma}
\begin{proof}
The proofs of the convergences \eqref{i_lemma1}, \eqref{ii_lemma1}
and \eqref{iii_lemma1} are given in  \cite[Lemma 2.2]{EE},
\cite[Lemma 2.1]{EEO} and \cite[Lemma 2.2]{EEO}, respectively. For
\eqref{cv_fBm}, since the functions $\phi_i$, $\ff$, are bounded, we
have if $\frac12<H<1$,
\begin{eqnarray*}
E\left[\left(\frac{1}{n}\int_0^n \phi_i(s)dB^H_s\right)^2\right]&=&\frac{H(2H-1)}{n^2}\int_0^n\int_0^n \phi_i(u)\phi_i(v) \vert u-v \vert^{2H-2} du dv\\
  &\leq & C\frac{H(2H-1)}{n^2}\int_0^n\int_0^n \vert u-v \vert^{2H-2} du dv=C
  n^{-2(1-H)},
\end{eqnarray*}
if $H=\frac12$,
\begin{eqnarray*}
E\left[\left(\frac{1}{n}\int_0^n
\phi_i(s)dB^H_s\right)^2\right]&=&\frac{1}{n^2}\int_0^n  \phi_i^2(u)
du
  \leq C n^{-1}.
\end{eqnarray*}
Combining these facts together with  Gaussian property of
$\frac{1}{n}\int_0^n \phi_i(s)dB^H_s$ and  Lemma \ref{Borel-lemma},
the proof is finished.
\end{proof}
\begin{lemma}\label{Lemma2}
Let  $\{A_t,\geq 0\}$   be given by \eqref{zeta}. Then, for every
$t\geq 1$,
\begin{equation}
A_t=\frac{A_1}{1-e^{-\alpha}}+R_t, \label{i_lemma2}
\end{equation}
where
\begin{equation*}
R_t:=-\frac{A_1e^{-\alpha[t]}}{1-e^{-\alpha}}+e^{-\alpha[t]}\int_0^{t-[t]}e^{-\alpha
s} L(s)ds, \text{ with $[t]$ is the integer part of $t$}.
\end{equation*}
Moreover, there is a positive constant  $C$ depending only on
$\alpha$ such that, for every $t\geq 1$,
\begin{equation}
\vert R_t\vert\leq C e^{-\alpha t}.\label{ii_lemma2}
\end{equation}
As a consequence, as $t\rightarrow\infty$,
\begin{equation}
A_t\longrightarrow
A_{\infty}:=\frac{A_1}{1-e^{-\alpha}}.\label{iii_lemma2}
\end{equation}
Also, for every $\ff$,
\begin{eqnarray}
\lim_{n\rightarrow\infty}e^{-\alpha n}\int_0^n \phi_i(s) e^{\alpha
s} A_s  ds= \lambda_{\phi_i}A_{\infty}, \ \ \text { where }\
\lambda_{\phi_i}:=\frac{\int_0^1 \phi_i(s)e^{\alpha
s}ds}{e^{\alpha}-1}.\label{iv_lemma2}
\end{eqnarray}
\end{lemma}
\begin{proof}
Notice that for every positive integer $n\geq1$,
\begin{eqnarray*}
A_n&=&\int_0^n e^{-\alpha s}L(s)ds\\ &=&\sum_{k=0}^{n-1}\int_k^{k+1}e^{-\alpha s}L(s)ds\\
   &=&\sum_{k=0}^{n-1}e^{-\alpha k}\int_0^1 e^{-\alpha s}L(s)ds\\
   &=& A_1\times\frac{1-e^{-\alpha n}}{1-e^{-\alpha}},
\end{eqnarray*}
where in the third equality we used the fact that $L$ is
$1$-periodic. Hence, for every $t\geq 1$,
\begin{eqnarray*}
A_t&=&\int_0^te^{-\alpha s}L(s)ds\\
&=& A_{[t]}+\int_{[t]}^t
e^{-\alpha s}L(s)ds\\
   &=& A_1\times\frac{1-e^{-\alpha [t]}}{1-e^{-\alpha}}+e^{-\alpha[t]}\int_0^{t-[t]}e^{-\alpha s} L(s)ds\\
   &=&\frac{A_1}{1-e^{-\alpha}}+R_t,
\end{eqnarray*}
which implies \eqref{i_lemma2}. On the other hand, since  $L$ is
bounded and $0\leq t-[t]<1$, we have
\begin{eqnarray*}
\vert R_t \vert&=& \left\vert-\frac{A_1 e^{-\alpha[t]}}{1-e^{-\alpha}}+e^{-\alpha [t]}\int_0^{t-[t]} e^{-\alpha s}L(s)ds\right\vert\\
              &\leq& C e^{-\alpha [t]}+ C e^{-\alpha [t]}\int_0^{t-[t]}e^{-\alpha s} ds\\
              &\leq&C e^{-\alpha t}+ C e^{-\alpha t}\int_0^{1}e^{-\alpha s} ds\\
             &\leq & C  e^{-\alpha t},
\end{eqnarray*}
which proves \eqref{ii_lemma2}. Furthermore, the convergence
\eqref{iii_lemma2} is
immediately obtained from \eqref{i_lemma2} and \eqref{ii_lemma2}.\\
Let us now prove \eqref{iv_lemma2}. According to  \eqref{i_lemma2},
we have, for every $\ff$,
\begin{equation*}
e^{-\alpha n}\int_0^n \phi_i(s)e^{\alpha s} A_s
ds=A_{\infty}e^{-\alpha n}\int_0^n \phi_i(s)e^{\alpha s}ds+e^{\alpha
n}\int_0^n e^{\alpha s}\phi_i(s)R(s) ds.
\end{equation*}
Using \eqref{ii_lemma2} and the boundedness of the functions
$\phi_i$, $\ff$, we deduce  that,  as $n\rightarrow\infty$,
\[e^{-\alpha n}\int_0^n e^{\alpha s}\phi_i(s)R(s) ds\longrightarrow
0.\] Moreover, since the functions $\phi_i$, $\ff$, are
$1$-periodic,
\begin{eqnarray}
\lim_{n\rightarrow\infty}e^{-\alpha n}\int_0^n\phi_i(s)e^{\alpha s}
ds&=&\lim_{n\rightarrow\infty}e^{-\alpha
n}\sum_{k=0}^{n-1}\int_k^{k+1}\phi_i(s)e^{\alpha s}
ds\nonumber\\&=&\lim_{n\rightarrow\infty}e^{-\alpha
n}\sum_{k=0}^{n-1}e^{\alpha k}\int_0^1 \phi_i(s)e^{\alpha s}
ds\nonumber\\&=&\lim_{n\rightarrow\infty}\frac{1-e^{-\alpha
n}}{e^{\alpha}-1} \int_0^1 \phi_i(s)e^{\alpha s}ds\nonumber\\&=&
\frac{\int_0^1 \phi_i(s)e^{\alpha s}ds}{e^{\alpha}-1}.\label{cv-phi}
\end{eqnarray}
\end{proof}

\begin{lemma}\label{Lemma2-duplicate} Let $\{X_t,\geq 0\}$ and $\{Z_t,\geq 0\}$ be
given by \eqref{PeriodicMean} and \eqref{exp-Z}, respectively. Then,
for every $\ff$, we have almost surely, as $n\rightarrow\infty$,
\begin{eqnarray}
e^{-\alpha n}\int_0^n \phi_i(s)e^{\alpha s}Z_s ds&\longrightarrow&\lambda_{\phi_i}Z_{\infty},\label{v_lemma2}\\
ne^{-\alpha n}\Lambda_{n,i}=e^{-\alpha n}\int_0^n \phi_i(s)X_s ds&\longrightarrow&\lambda_{\phi_i}(A_{\infty}+\alpha Z_{\infty}),\label{vi_lemma2}\\
e^{-\alpha n}\int_0^n
\phi_i(s)dX_s&\longrightarrow&\alpha\lambda_{\phi_i}(A_{\infty}+\alpha
Z_{\infty}),\label{vii_lemma2}
\end{eqnarray}
where $Z_{\infty}$, $A_{\infty}$ and $\lambda_{\phi_i}$ are defined
by \eqref{ii_lemma1}, \eqref{iii_lemma2} and \eqref{iv_lemma2},
respectively.
\end{lemma}
\begin{proof}
First we prove \eqref{v_lemma2}. We have, for every $\ff$,
\begin{eqnarray*}
e^{-\alpha n}\int_0^n \phi_i(s)e^{\alpha s}Z_s ds &=& Z_{\infty} e^{-\alpha n}\int_0^n \phi_i(s)e^{\alpha s}ds+e^{-\alpha n}\int_0^n e^{\alpha s}\phi_i(s)(Z_s-Z_{\infty})ds\\
                                                  &=& Z_{\infty}I_{i,n}+J_{i,n}.
 \end{eqnarray*}
Using \eqref{cv-phi}, we obtain $I_{i,n}\longrightarrow
\lambda_{\phi_i}$, as $n\rightarrow\infty$. Further,
\begin{eqnarray*}
\vert J_{i,n}\vert &\leq & e^{-\alpha n}\int_0^n e^{\alpha s}\vert\phi_i(s)\vert \vert Z_s-Z_{\infty}\vert ds\\
                &\leq & Ce^{-\alpha n}\int_0^n e^{\alpha s}\int_s^{\infty}e^{-\alpha r} \vert B^H_r \vert dr ds\\
                &\leq &Ce^{-\alpha n}\int_0^n e^{\frac{\alpha}{2} s}\int_s^{\infty}e^{-\frac{\alpha}{2} r} \vert B^H_r \vert dr ds\\
                && \longrightarrow0
\end{eqnarray*}
almost surely,  as $n\rightarrow\infty$, using \eqref{i_lemma1}. This  completes the proof of \eqref{v_lemma2}.\\
On the other hand, by \eqref{form_X_t}, we can write, for every
$\ff$,
\begin{eqnarray*}
e^{-\alpha n}\int_0^n \phi_i(s) X_s ds&=&e^{-\alpha n}\int_0^n \phi_i(s)e^{\alpha s}A_s ds+ \alpha e^{-\alpha n}\int_0^n \phi_i(s)e^{\alpha s}Z_sds+e^{-\alpha n}\int_0^n \phi_i(s)B_s^Hds\\
                                     &=& A_{i,n,1}+A_{i,n,2}+A_{i,n,3}.
\end{eqnarray*}
By \eqref{iv_lemma2} and \eqref{v_lemma2}, we obtain, for every
$\ff$,
$A_{i,n,1}+A_{i,n,2}\longrightarrow\lambda_{\phi_i}\left(A_{\infty}+\alpha
Z_{\infty}\right)$ almost surely, as $n\rightarrow\infty$. Moreover,
 for every $\ff$, $A_{i,n,3}\longrightarrow0$ almost surely, as
 $n\rightarrow\infty$, thanks to \eqref{i_lemma1} and the fact that
 the functions $\phi_i$, $\ff$, are bounded. Thus, the convergence
 \eqref{vi_lemma2} is obtained.\\
For \eqref{vii_lemma2}, we have, for every $\ff$,
\begin{eqnarray*}
e^{-\alpha n}\int_0^n\phi_i(s)dX_s&=& e^{-\alpha n}\sum_{i=1}^p
\mu_i\int_0^n \phi^2_i(s) ds +\alpha e^{-\alpha n}\int_0^n
\phi_i(s)X_s ds+ e^{-\alpha n}\int_0^n\phi_i(s)dB^H_s,
\end{eqnarray*}
which converges to  $\alpha\lambda_{\phi_i}(A_{\infty}+\alpha
Z_{\infty})$ almost surely, as
 $n\rightarrow\infty$,  according to \eqref{cv_fBm} and
 \eqref{vi_lemma2}.
Then  the proof of Lemma \ref{Lemma2-duplicate} is completed.
\end{proof}
\begin{lemma}\label{Lemma3}
The following convergences hold almost surely, as
$n\rightarrow\infty$,
\begin{eqnarray}
e^{-\alpha n} X_n&\longrightarrow& A_{\infty}+\alpha Z_{\infty},\label{i_lemma3}\\
e^{-\alpha n}\int_0^n X_s ds &\longrightarrow& \frac{A_{\infty}}{\alpha}+Z_{\infty},\label{ii_lemma3}\\
e^{-2\alpha n}\int_0^n X_s^2 ds &\longrightarrow& \frac{1}{2\alpha}(A_{\infty}+\alpha Z_{\infty})^2,\label{iii_lemma3}\\
\frac{e^{-\alpha n}}{n}\int_0^n \vert B_s^H X_s\vert ds  &\longrightarrow& 0,\label{iv_lemma3}\\
\frac{e^{-\alpha n}}{n}\int_0^n \tilde{L}(s)X_s ds&\longrightarrow&
\left(\frac{A_{\infty}}{\alpha}+Z_{\infty}\right)\int_0^1L(s)ds,\label{v_lemma3}
\end{eqnarray}
where $\tilde{L}(t)$ is defined by \eqref{def_L_Sigma}.
\end{lemma}
\begin{proof}
The statements  \eqref{i_lemma3}---\eqref{iv_lemma3} can be
immediately obtained from \eqref{form_X_t}, \eqref{i_lemma1},
\eqref{ii_lemma1}, \eqref{iii_lemma2} and L'H\^opital's rule. It
remains to prove \eqref{v_lemma3}. Notice that
 \begin{eqnarray*}
  \tilde{L}(t)&=&\int_0^{[t]}L(s)ds+\int_{[t]}^t L(s)ds\\
              &=&\sum_{k=0}^{[t]-1}\int_k^{k+1} L(s)ds+\int_{[t]}^t L(s)ds\\
              &=&[t]\int_0^1 L(s) ds+\int_{[t]}^t L(s) ds\\
              &=& t\int_0^1 L(s) ds+\left(([t]-t)\int_0^1 L(s) ds+\int_{[t]}^t L(s)ds\right)\\
              &=:& t\int_0^1 L(s) ds+l(t).
\end{eqnarray*}
Moreover, it is clear that  the function $l(t)$ is bounded. So,
\begin{equation*}
\frac{e^{-\alpha n}}{n}\int_0^n \tilde{L}(s)X_s ds=\frac{e^{-\alpha
n}}{n}\int_0^n sX_s ds \left(\int_0^1 L(r)dr\right)+\frac{e^{-\alpha
n}}{n}\int_0^n l(s)X_s ds.
\end{equation*}
Furthermore, using \eqref{i_lemma3} and  L'H\^opital's rule, we
obtain
\begin{equation*}
\frac{e^{-\alpha n}}{n}\left\vert\int_0^n l(s) X_s ds\right\vert\leq
C \frac{e^{-\alpha n}}{n}\int_0^n \vert X_s \vert ds\longrightarrow
0,
\end{equation*}
and
\begin{equation*}
\frac{e^{-\alpha n}}{n}\int_0^n s X_s ds\rightarrow \left(
\frac{A_{\infty}}{\alpha}+Z_{\infty}\right)
\end{equation*}
almost surely, as $n\rightarrow\infty$. Then the desired result is
obtained.
\end{proof}

\section{Strong consistency}
Here we prove the strong consistency of the LSE $\hat{\theta}_n$, as
$n\rightarrow\infty$.
\begin{theorem}Assume $\frac12\leq H<1$. Then, almost surely, as $n\rightarrow
\infty$,
\[\hat{\theta}_n=(\hat{\mu}_n,
\hat{\alpha}_{n})\longrightarrow\theta=(\mu,\alpha).\]
\end{theorem}
\begin{proof}
Suppose $\frac12<H<1$. Then, using   $\int_{0}^{n} X_{s} d
X_{s}=\frac{1}{2}X_{n}^{2}$ and \eqref{estimator1}, we can write
\begin{eqnarray*}
\hat{\alpha}_n&=&\dfrac{-e^{-2\alpha n}\sum_{k=1}^p \Lambda_{n,k}\int_0^n \phi_k(t)dX_t+\frac{e^{-2\alpha n}}{2}X_n^2}{e^{-2\alpha n}\int_0^n X_s^2 ds - n\sum_{k=1}^p e^{-2\alpha n}\Lambda_{k,n}^2}\\
           &=&\dfrac{-\sum_{k=1}^p\frac{1}{n}\left(e^{-\alpha n}\int_0^n \phi_k(s)X_s ds\right)
           \left(e^{-\alpha n}\int_0^n \phi_k(s) dX_s\right)
           +\frac{e^{-2\alpha n}}{2}X_n^2}{e^{-2\alpha n}\int_0^n X_s^2 ds
           -\frac{1}{n}\sum_{k=1}^p\left(e^{-\alpha n}\int_0^n \phi_k(s)X_s
           ds\right)^2}.
\end{eqnarray*}
Combining this with  \eqref{vi_lemma2}, \eqref{vii_lemma2},
\eqref{i_lemma3} and \eqref{iii_lemma3}, we deduce that
$\hat{\alpha}_n$ converges to $\alpha$, almost surely as
$n\rightarrow\infty$. Hence, it remains to prove the strong
consistency for each $\hat{\mu}_{n,i}$, $\ff$. It follows from
\eqref{estimator2} that $\hat{\mu}_{n,i}$, $\ff$ can be written as
\begin{equation}
\hat{\mu}_{n,i}=\frac{\frac{1}{n}I_{i,n,1}+I_{i,n,2}}{I_{n,3}}\label{mu_hat},
\end{equation}
with
\begin{eqnarray*}
I_{i,n,1}&=&\int_0^n\phi_i(s)dX_s\int_0^n X_s^2 ds-\frac{X_n^2}{2}\int_0^n\phi_i(s)X_s ds,\\
I_{i,n,2}&=&\Lambda_{n,i}\sum_{k=1}^p\Lambda_{n,k}\int_0^{n}\phi_k(s)
dX_s-
\left(\int_0^n\phi_i(s)dX_s\right)\sum_{k=1}^p \Lambda_{n,k}^2,\\
I_{n,3}&=&\int_0^n X^2_s ds-n\sum_{k=1}^p\Lambda_{n,k}^2.
\end{eqnarray*}
According to the convergences \eqref{vi_lemma2} and
\eqref{iii_lemma3}, we have almost surely, as $n\rightarrow\infty$,
\begin{eqnarray}
e^{-2\alpha n}I_{n,3}&=& e^{-2\alpha n}\int_0^n X_s^2 ds-n\sum_{k=1}^p e^{-2\alpha n}\Lambda_{n,k}^2\nonumber\\
                                 &=&e^{-2\alpha n}\int_0^n X_s^2 ds-\frac{1}{n}\sum_{k=1}^p\left(e^{-\alpha n}\int_0^n \phi_i(s)X_s ds\right)^2\nonumber\\
                                 &&\longrightarrow\frac{(A_{\infty}+\alpha
                                 Z_{\infty})^2}{2\alpha}.\label{denominator}
\end{eqnarray}
 Then, it remains to prove the following convergence,  for every $\ff$,
 \begin{equation*}
\frac{ e^{-2\alpha n}}{n}I_{i,n,1}+ e^{-2\alpha
n}I_{i,n,2}\longrightarrow\frac{1}{2\alpha}(A_{\infty}+\alpha
Z_{\infty})^2\mu_i
 \end{equation*}
almost surely, as $n\rightarrow\infty$.  Using \eqref{form_X_Sigma},
we get for $\ff$,
\begin{eqnarray*}
I_{i,n,1}&=&\int_0^n\phi(s)dX_s\left[\int_0^n\tilde{L}(s)d\Sigma_s+\frac{\alpha}{2} \Sigma_n^2+\int_0^n B^H_s d\Sigma_s \right]
\\
       &&-\frac{1}{2}\left(\tilde{L}(n)+\alpha\Sigma_n+B^H_n\right)^2\int_0^n\phi_i(s)d\Sigma_s.
\end{eqnarray*}
Moreover, since
\begin{equation*}
\int_0^n
\phi_i(s)d\Sigma_s=\frac{1}{\alpha}\int_0^n\phi_i(s)dX_s-\frac{1}{\alpha}\int_0^n\phi_i(s)L(s)ds-\frac{1}{\alpha}\int_0^n\phi_i(s)dB^H_s,
\end{equation*}
we obtain
\begin{eqnarray*}
I_{i,n,1}&=&\int_0^n\phi_i(s)dX_s\left(\int_0^n \tilde{L}(s)d\Sigma_s+\int_0^n B^H_s d\Sigma_s\right)\\
          &&-\frac{1}{2}\left[(\tilde{L}(n)+B^H_n)^2+2\alpha\Sigma_n(\tilde{L}(n)+B^H_n)\right]\int_0^n\phi_i(s)d\Sigma_s
          \\
          &&-\frac{1}{2}(\alpha\Sigma_n)^2\left(-\frac{1}{\alpha}\int_0^n\phi_i(s)L(s)ds-\frac{1}{\alpha}\int_0^n\phi_i(s)dB_s^H\right)\\
          &=&J_{i,n,1}+J_{i,n,2}+J_{i,n,3}+J_{i,n,4}+J_{i,n,5}.
\end{eqnarray*}
Using \eqref{vii_lemma2} and \eqref{v_lemma3}, we have almost
surely, as $n\rightarrow\infty$,
\begin{eqnarray*}
\frac{e^{-2\alpha n}}{n}J_{i,n,1}&=&e^{-\alpha n}\int_{0}^n\phi_i(s) dX_s\left(\frac{e^{-\alpha n}}{n}\int_0^n \tilde{L}(s)X_s ds\right)\\
                                 &&\longrightarrow\alpha\lambda_{\phi_i}(A_{\infty}+\alpha Z_{\infty})\left(\frac{A_{\infty}}{\alpha}+Z_{\infty}\right)\int_0^1 L(s)ds.
\end{eqnarray*}
By \eqref{vii_lemma2}and \eqref{iv_lemma3}, we get almost surely, as
$n\rightarrow\infty$,
\begin{eqnarray*}
\frac{e^{-2\alpha n}}{n}J_{i,n,2}&=&e^{-\alpha n}\int_0^n \phi(s)dX_s\left(\frac{e^{-\alpha n}}{n}\int_0^n B^H_s X_sds\right)\\
                                 &&\longrightarrow 0.
\end{eqnarray*}
The convergences \eqref{i_lemma1} and \eqref{vi_lemma2} imply that
almost surely, as $n\rightarrow\infty$,
\begin{eqnarray*}
\frac{e^{-2\alpha n}}{n}J_{i,n,3}&=&-\frac{1}{2}\frac{e^{-\alpha n}}{n}(\tilde{L}(n)+B^H_n)^2 e^{-\alpha n}\int_0^n\phi_i(s)X_s ds\\
                                  &&\longrightarrow 0.
\end{eqnarray*}
According to \eqref{ii_lemma3}, \eqref{vii_lemma2} and
\eqref{i_lemma1}, we get almost surely, as $n\rightarrow\infty$,
\begin{eqnarray*}
\frac{e^{-2\alpha n}}{n}J_{i,n,4}&=&-\alpha \left(e^{-\alpha n}\int_0^n X_s \right)\left(\frac{1}{n}\tilde{L}(n)\right)\left(e^{-\alpha n}\int_0^n \phi_i(s)X_s ds\right)-\\
                                 &&\alpha \left(\frac{e^{-\alpha n}}{n} B^H_n\right)\left(e^{-\alpha n}\int_0^n \phi_i(s)X_s ds\right)\\
                                 &&\longrightarrow-\alpha\lambda_{\phi_i}(A_{\infty}+\alpha Z_{\infty})\left(\frac{A_{\infty}}{\alpha}+Z_{\infty}\right)\int_0^1 L(s)ds.
\end{eqnarray*}
Furthermore, \eqref{ii_lemma3}  and the periodicity property yield
almost surely, as $n\rightarrow\infty$,
\begin{eqnarray*}
\frac{e^{-2\alpha n}}{n}J_{i,n,5}&=&\frac{\alpha}{2}\left(e^{-\alpha n}\int_0^n X_s ds\right)^2 \frac{1}{n}\int_0^n \phi_i(s)L(s)ds\\
                                 &&\longrightarrow\frac{1}{2\alpha}(A_{\infty}+\alpha Z_{\infty})^2\mu_i.
\end{eqnarray*}
Also, \eqref{ii_lemma3} and \eqref{cv_fBm}  imply that almost
surely, as $n\rightarrow\infty$,
\begin{eqnarray*}
\frac{e^{-2\alpha n}}{n}J_{i,n,6}&=&\frac{\alpha}{2}\left(e^{-\alpha n}\int_0^n X_s ds\right)^2\frac{1}{n}\int_0^n \phi_i(s)dB^H_s\\
                                  &&\longrightarrow 0.
\end{eqnarray*}
Consequently, as $n\rightarrow\infty$,
\begin{eqnarray}
\frac{e^{-2\alpha
n}}{n}I_{i,n,1}\longrightarrow\frac{1}{2\alpha}(A_{\infty}+\alpha
Z_{\infty})^2\mu_i\label{numerateur1}
\end{eqnarray} almost surely for every $\ff$.\\
Finally, since $\int_0^n \phi^2 _i(s)ds=n$ for $\ff$, then
\begin{eqnarray*}
I_{i,n,2}&=&\Lambda_{n,i}\sum_{k=1}^p\Lambda_{n,k}\left(n\mu_k+\alpha
n\Lambda_{n,k}+\int_0^n \phi_k(s)
dB^H_s\right)\\&&-\left(n\mu_i+\alpha n\Lambda_{n,i}+\int_0^n
\phi_i(s)
dB^H_s\right)\sum_{k=1}^n \Lambda_{n,k}^2\\
         &=&\Lambda_{n,i}\sum_{k=1}^p\Lambda_{n,k}\left(n\mu_k+\int_0^n \phi_k(s)
dB^H_s\right)-\left(n\mu_i+\int_0^n \phi_i(s)
dB^H_s\right)\sum_{k=1}^n\Lambda_{n,k}^2,
\end{eqnarray*}
where we used
\begin{equation*}
\int_0^n\phi_i(s)dX_s=\int_0^n\mu_i\phi_i(s)^2 ds+\alpha n
\Lambda_{n,i}+\int_0^n \phi_i(s) dB^H_s.
\end{equation*}
By  \eqref{cv_fBm}  and \eqref{vi_lemma2}, we get for every $\ff$,
almost surely, as $n\rightarrow\infty$,
\begin{equation}
e^{-2\alpha n}I_{i,n,2}\longrightarrow 0\label{numerateur2}.
\end{equation}
Therefore, the facts \eqref{mu_hat}, \eqref{denominator},
\eqref{numerateur1} and \eqref{numerateur2} achieve the proof of
desired
result.\\
For the case $H=\frac12$, it suffices  to use $\int_{0}^{n} X_{s} d
X_{s}=\frac{1}{2}\left(X_{n}^{2}-n\right)$ and the same arguments as
above.
\end{proof}

\section{Asymptotic distribution}

 In order to   investigate  the
asymptotic behavior in distribution of the estimator
$\widehat{\theta}_n$, as $n\rightarrow\infty$, we will need the
following lemmas.
\begin{lemma}\label{lemma-decomp-S}Assume that $\frac12\leq H<1$. Then,  for every $n\geq0$,
\begin{eqnarray}
\frac12X_n^2-\sum_{k=1}^p \Lambda_{n, k} \int_{0}^{n}
\phi_k(s)dX_s&=&n\alpha \gamma_n^{-1}  +\left(A_{\infty} +\alpha Z_n
\right) \int_0^n e^{\alpha s} dB^H_s+S_n,\label{numerator-alpha}
\end{eqnarray} where $Z_n$ is given by (\ref{exp-Z}), and the sequence $S_n$ is defined by
\begin{eqnarray*}S_n&:=& \frac12  (B_n^H)^2+e^{n \alpha}R_nB_n^H-\int_0^n L(s) B_s^H
ds-\alpha \int_0^n (B_s^H)^2 ds-\sum_{k=1}^p \Lambda_{n, k}
\int_{0}^{n} \phi_k(s)dX_s\\&&-\sum_{k=1}^p \Lambda_{n, k}
\int_{0}^{n} \phi_k(s)dB^H_s-\alpha\int_0^n e^{\alpha s}B_s^H R_s ds
+\alpha^2 \int_0^n e^{-\alpha s} B^H_s\int_0^s e^{\alpha r} B^H_r dr
ds.
\end{eqnarray*}
Moreover, as $n\longrightarrow\infty$,
\begin{eqnarray}e^{-\alpha n}S_n\longrightarrow0\ \mbox{ almost surely. }\label{cv-S}
\end{eqnarray}
\end{lemma}
\begin{proof}Let us prove (\ref{numerator-alpha}). First define the
process
\begin{eqnarray*} M_t:=\int_0^te^{\alpha s}B^H_s ds,\quad t\geq0.
\end{eqnarray*}
According to (\ref{form_X_Sigma}), we have \begin{eqnarray} \frac12
X_t^2=\frac12(\tilde{L}(t))^2+\frac12 (B_t^H)^2
+\tilde{L}(t)B_t^H+\frac12\alpha^2\Sigma_t^2+\alpha\tilde{L}(t)\Sigma_t+\alpha\Sigma_tB_t^H.\label{equaX1}
\end{eqnarray}
Furthermore,  using (\ref{IBP}) and (\ref{form_X_Sigma})
\begin{eqnarray}\frac{\alpha^2}{2}\Sigma_t^2
&=&\alpha^2\int_0^t\Sigma_sd\Sigma_s\nonumber\\&=&\alpha^2\int_0^t\Sigma_sX_sds
\nonumber\\&=&\alpha\int_0^t X_s^2ds-\alpha\int_0^t\tilde{L}(s)
d\Sigma_s-\alpha\int_0^tB^H_sX_sds\nonumber
\\&=&\alpha\int_0^t X_s^2ds-\alpha \tilde{L}(t)\Sigma_t+
\alpha\int_0^t \Sigma_s
L(s)ds-\alpha\int_0^tB^H_sX_sds.\label{equaX2}
\end{eqnarray}
Moreover, it follows from (\ref{form_X_t}) that
\begin{eqnarray}-\alpha\int_0^tB^H_sX_sds &=&-\alpha\int_0^tB^H_se^{\alpha s}A_sds-\alpha^2\int_0^tB^H_se^{\alpha s}Z_sds
-\alpha\int_0^t(B^H_s)^2ds\nonumber\\&=&
-\alpha\int_0^tB^H_se^{\alpha
s}A_sds-\alpha\int_0^t(B^H_s)^2ds-\alpha^2\left[M_tZ_t-\int_0^tM_sdZ_s
\right]\nonumber
\\&=&
-\alpha\int_0^tB^H_se^{\alpha
s}A_sds-\alpha\int_0^t(B^H_s)^2ds-\alpha^2M_tZ_t\label{equaX3}\\&&+\alpha^2\int_0^te^{-\alpha
s}B^H_s\int_0^se^{\alpha r}B^H_rdrds.\nonumber
\end{eqnarray}
On the other hand, by (\ref{form_X_Sigma}) and (\ref{form_X_t}), we
have
\begin{eqnarray}\alpha\Sigma_tB_t^H&=&B_t^H(X_t-\tilde{L}(t)-B_t^H)\nonumber\\
&=&B_t^H\left(-\tilde{L}(t)+e^{\alpha t}A_t+\alpha e^{\alpha
t}Z_t\right).\label{equaX4}
\end{eqnarray}
Combining (\ref{equaX1}), (\ref{equaX2}), (\ref{equaX3})  and
(\ref{equaX4}), we obtain
\begin{eqnarray} \frac12
X_t^2&=&\frac12(\tilde{L}(t))^2+\frac12 (B_t^H)^2+ \alpha\int_0^t
X_s^2ds+ \alpha\int_0^t \Sigma_s L(s)ds-\alpha\int_0^tB^H_se^{\alpha
s}A_sds\label{equaX5}
\\&&-\alpha\int_0^t(B^H_s)^2ds-\alpha^2M_tZ_t+\alpha^2\int_0^te^{-\alpha s}B^H_s\int_0^se^{\alpha
r}B^H_rdrds+e^{\alpha t}B_t^H A_t+\alpha e^{\alpha t}B_t^H
Z_t.\nonumber
\end{eqnarray}
Further,
\begin{eqnarray} -\alpha^2M_tZ_t+\alpha e^{\alpha t}B_t^H
Z_t&=&-\alpha Z_t\left(\alpha M_t-e^{\alpha
t}B_t^H\right)\nonumber\\
&=&\alpha Z_t\int_0^t e^{\alpha s}dB_s^H.\label{equaX6}
\end{eqnarray}
Also, using (\ref{i_lemma2}),
\begin{eqnarray} e^{\alpha t}B_t^H A_t-\alpha\int_0^tB^H_se^{\alpha
s}A_sds &=& A_{\infty}e^{\alpha t}B_t^H+e^{\alpha t}B_t^HR_t-\alpha
A_{\infty}\int_0^tB^H_se^{\alpha s}
ds -\alpha\int_0^tB^H_se^{\alpha s}R_sds\nonumber\\
&=& A_{\infty}\int_0^te^{\alpha s}B_s^Hds+e^{\alpha t}B_t^HR_t
-\alpha\int_0^tB^H_se^{\alpha s}R_sds.\label{equaX7}
\end{eqnarray}
Combining (\ref{equaX5}), (\ref{equaX6})   and (\ref{equaX7}), we
get
\begin{eqnarray} \frac12
X_t^2&=&\frac12(\tilde{L}(t))^2+\frac12 (B_t^H)^2+ \alpha\int_0^t
X_s^2ds+ \alpha\int_0^t \Sigma_s L(s)ds -\alpha\int_0^t(B^H_s)^2ds
\label{equaX8}
\\&&+\alpha^2\int_0^te^{-\alpha s}B^H_s\int_0^se^{\alpha r}B^H_rdrds
+\alpha Z_t\int_0^t e^{\alpha s}dB_s^H+A_{\infty}\int_0^te^{\alpha
s}B_s^Hds+e^{\alpha t}B_t^HR_t\nonumber\\&&
-\alpha\int_0^tB^H_se^{\alpha s}R_sds.\nonumber
\end{eqnarray}
On the other hand, using (\ref{PeriodicMean}),
\begin{eqnarray}
 \sum_{k=1}^p \Lambda_{n, k} \int_{0}^{n} \phi_k(s)dX_s&=&\sum_{k=1}^p \mu_k n\Lambda_{n, k}+\sum_{k=1}^p \alpha n\Lambda_{n,
 k}^2+\sum_{k=1}^p \Lambda_{n, k}\int_{0}^{n} \phi_k(s)dB^H_s\nonumber\\
 &=&\int_0^nL(s)X_sds+\sum_{k=1}^p \alpha n\Lambda_{n,
 k}^2+\sum_{k=1}^p \Lambda_{n, k}\int_{0}^{n} \phi_k(s)dB^H_s.\label{equaX9}
\end{eqnarray}
Further, by (\ref{form_X_Sigma}) and (\ref{IBP}),
\begin{eqnarray}
 \alpha\int_0^t \Sigma_s L(s)ds -\int_0^nL(s)X_sds&=&\int_0^t\left( \alpha\Sigma_s  - X_s\right)ds\nonumber\\
 &=&\int_0^tL(s)\tilde{L}(s)ds-\int_0^tL(s)B_s^Hds\nonumber\\
 &=&\frac12(\tilde{L}(t))^2-\int_0^tL(s)B_s^Hds.\label{equaX10}
\end{eqnarray}
Consequently, the equalities (\ref{equaX8}), (\ref{equaX9})   and
(\ref{equaX10}) lead to (\ref{numerator-alpha}). \\
Finally, the convergence (\ref{cv-S}) is a direct consequence of
\eqref{i_lemma1}, \eqref{ii_lemma2}, \eqref{vi_lemma2} and
\eqref{cv_fBm}.
\end{proof}

\begin{lemma}Assume $\frac12\leq H<1$. Then, we have
\begin{align}\label{estimator-alpha}
\hat{\alpha}_{n}-\alpha= \frac{\left(A_{\infty} +\alpha Z_n \right)
\int_0^n e^{\theta s} dB^H_s+S_{n,H}}{n\gamma_n^{-1}},
\end{align}
where $S_{n,H}=S_n-\frac{n}{2}$ if $H=\frac12$, and $S_{n,H}=S_n$ if
$\frac12<H<1$,  with $S_n$ is defined in Lemma \ref{lemma-decomp-S}.
 Moreover,  we have
\begin{eqnarray}\hat{\mu}_n-\mu=(\alpha-\hat{\alpha}_{n})(\Lambda_{n,1}, \ldots,
\Lambda_{n,p})+\frac{G_n}{n},\label{estimator-mu}
\end{eqnarray}
where
\begin{eqnarray}G_n:=\left( \int_0^n\phi_1(s)dB^H_s,
\ldots,\int_0^n\phi_p(s)dB^H_s\right).\label{def-G}
\end{eqnarray}
\end{lemma}
\begin{proof}The representation \eqref{estimator-alpha} follows
directly from   \eqref{estimator1} and \eqref{numerator-alpha}. For
\eqref{estimator-mu}, according to \eqref{estimator1} and
\eqref{estimator2}, we can write, for every $i=1,\ldots,p$,
\begin{eqnarray*}
\hat{\mu}_{n, i}&=\frac{1}{n}\int_{0}^{n} \phi_{i}(s)dX_s
-\Lambda_{n, i} \hat{\alpha}_{n}.
\end{eqnarray*}
This combined with \eqref{PeriodicMean} imply that, for every
$i=1,\ldots,p$,
\begin{eqnarray*}
\hat{\mu}_{n, i}&=&\frac{1}{n}\int_{0}^{n}
 \mu_i\phi_i^2(s)ds+\frac{\alpha}{n}\int_{0}^{n}
\phi_{i}(s)X_s ds +\frac{1}{n}\int_{0}^{n} \phi_{i}(s)dB^H_s
-\Lambda_{n, k} \hat{\alpha}_{n}\\&=&
 \mu_i +\alpha \Lambda_{n, i} +\frac{1}{n}\int_{0}^{n} \phi_{i}(s)dB^H_s
-\Lambda_{n, i} \hat{\alpha}_{n},
\end{eqnarray*}
 which completes the proof.
\end{proof}

\begin{lemma}\label{lemma-cv-couple-law} Assume $\frac12\leq H<1$, and let $G_n$ be given by \eqref{def-G}.
   Let $F$ be any
$\sigma\{B^H_t,t\geq0\}-$measurable random variable such that
$P(F<\infty)=1$. Then,  as $n\rightarrow\infty$,
\begin{eqnarray}\left(e^{-\theta T}\int_0^Te^{\theta
t}dB^H_t,F,\frac{G_n}{n^{H}}\right)\overset{Law}{\longrightarrow}\left(N_1,F,
N_2\right),\label{cv-law-couple-F-G}\end{eqnarray} where
$N_1\sim\mathcal{N}(0,\sigma_H^2)$, $N_2\sim\mathcal{N}(0,D)$ and
$B^H$ are independent, with the variance
$\sigma_H^2=\frac{H\Gamma(2H)}{\alpha^{2H}}$ and the covariance
matrix $D=\left(\int_{0}^{1} \int_{0}^{1} \phi_{i}(x) \phi_{j}(y) d
x d y\right)_{1\leq i,j\leq p}$.
\end{lemma}
\begin{proof} Using similar arguments as in the proof
of   \cite[Lemma 7]{EN} it suffices to prove that for every positive
integer $d$, and positive constants $s_1,\ldots,s_d$,
\begin{eqnarray*}\left(e^{-\alpha n}\int_0^ne^{\alpha
s}dB^H_s,B^H_{s_1},\ldots,B^H_{s_d},\frac{G_n}{n^{H}}\right)\overset{Law}{\longrightarrow}\left(N_1,B^H_{s_1},\ldots,B^H_{s_d},N_2\right)\end{eqnarray*}
as
 $n\rightarrow\infty$. Moreover, since the left-hand side in this latter convergence is a
Gaussian vector,   it is sufficient to establish the convergence of
its covariance matrix. From \cite[Lemma 6]{BEO} we have, for every
$1/2\leq H<1$,
\[\lim _{n \rightarrow \infty} E\left[\left(e^{-\alpha n} \int_{0}^{n}
e^{\alpha s} d B^H_{s}\right)^{2}\right]=\frac{H \Gamma(2
H)}{\alpha^{2 H}},\] and  for all fixed $s \geq 0$,
\[
\lim _{n \rightarrow \infty} E\left(B^H_{s} \times e^{-\alpha n}
\int_{0}^{n} e^{\alpha r} d B^H_{r}\right)=0.
\]
 Moreover, it follows from \cite{BEV2017} that,
for every $1/2\leq H<1$,  as $n \rightarrow \infty$
\[n^{-H} G_n \stackrel{\operatorname{law}}{\longrightarrow}
\mathcal{N}\left(0,D\right).
\]
Hence, to finish the proof it remains to check that, for all fixed
$s \geq 0$,
\begin{eqnarray}
\lim _{n \rightarrow \infty} E\left(B^H_{s} \times
\frac{1}{n^H}\int_{0}^{n} \phi_{i}(s)dB^H_s\right)=0,\quad
i=1,\ldots,p,\label{cv1-covariance}
\end{eqnarray}
and
\begin{eqnarray}
\lim _{n \rightarrow \infty} E\left( e^{-\alpha n} \int_{0}^{n}
e^{\alpha r} d B^H_{r} \times \frac{1}{n^H}\int_{0}^{n}
\phi_{i}(s)dB^H_s\right)=0,\quad i=1,\ldots,p.\label{cv2-covariance}
\end{eqnarray}
Let us prove \eqref{cv1-covariance}. Suppose $\frac12<H<1$. Fix $s
\geq 0$. We have, for every $i=1,\ldots,p$,
\begin{eqnarray*}
&& \left| E\left(B^H_{s} \times \frac{1}{n^H}\int_{0}^{n}
\phi_{i}(s)dB^H_s\right)\right|\\&\leq&\frac{H(2H-1)}{n^H}\int_0^ndv|\phi_{i}(v)|\int_0^sdu|u-v|^{2H-2}\\
&=&\frac{H(2H-1)}{n^H}\left(\int_0^sdv|\phi_{i}(v)|\int_0^sdu|u-v|^{2H-2}+\int_s^ndy|\phi_{i}(v)|\int_0^sdu(v-u)^{2H-2}\right)\\
&\leq&\frac{C}{n^H}\left(s^{2H}+\int_s^ndv\int_0^sdu(v-u)^{2H-2}\right)\\
&\leq&\frac{C}{n^H}\left(s^{2H}+s\int_s^ndv (v-s)^{2H-2}\right)\\
&=&C\left(\frac{s^{2H}}{n^H}+\frac{s}{(2H-1)n^H}  (n-s)^{2H-2}\right)\\
&&\longrightarrow0,
\end{eqnarray*}
as $n \rightarrow \infty$, since $H<1$.\\
For \eqref{cv2-covariance}, we have, for every $i=1,\ldots,p$,
\begin{eqnarray*}  &&\left| E\left( e^{-\alpha n} \int_{0}^{n} e^{\alpha r}
d B^H_{r} \times \frac{1}{n^H}\int_{0}^{n}
\phi_{i}(s)dB^H_s\right)\right| \\&\leq&\frac{H(2H-1)e^{-\alpha
n}}{n^H}\int_0^ndve^{\alpha v}\int_0^n|\phi_{i}(u)|du|u-v|^{2H-2}\\
&\leq&\frac{Ce^{-\alpha n}}{n^H}\int_0^ndve^{\alpha v}\int_0^n
du|u-v|^{2H-2}\\&=&C\left[\frac{e^{-\alpha
n}}{n^H}\int_0^ndve^{\alpha v}\int_0^v
du(v-u)^{2H-2}+\frac{e^{-\alpha n}}{n^H}\int_0^ndve^{\alpha
v}\int_v^n du(u-v)^{2H-2}\right].
\end{eqnarray*}
Further, by L'H\^opital's rule, \begin{eqnarray*} \lim _{n
\rightarrow \infty}\frac{e^{-\alpha n}}{n^H}\int_0^ndve^{\alpha
v}\int_0^v du(v-u)^{2H-2}&=&\lim _{n \rightarrow
\infty}\frac{e^{-\alpha n}}{(2H-1)n^H}\int_0^ndve^{\alpha v}
v^{2H-1}\\&=&\lim _{n \rightarrow \infty}  \frac{e^{\alpha n}
n^{2H-1}}{(2H-1)(n^H\alpha e^{\alpha n}+Hn^{H-1} e^{\alpha n})}
\\&=&\lim _{n \rightarrow \infty}  \frac{n^{H-1}}{(2H-1)(\alpha
+H/n)}\\
&=&0.
\end{eqnarray*}
Moreover, making the change of variables $x=n-v$, we obtain
 \begin{eqnarray*} \frac{e^{-\alpha n}}{n^H}\int_0^ndve^{\alpha
v}\int_v^n du(u-v)^{2H-2}&=& \frac{e^{-\alpha
n}}{(2H-1)n^H}\int_0^ne^{\alpha v} (n-v)^{2H-1}dv\\
&=& \frac{1}{(2H-1)n^H}\int_0^ne^{-\alpha x} x^{2H-1}dv\\
 &\leq&\frac{\Gamma(2H)}{(2H-1)n^H\alpha^{2H}}\\
 &&\longrightarrow0,
\end{eqnarray*}
as $n \rightarrow \infty$. The proof of \eqref{cv1-covariance} and
\eqref{cv2-covariance}, when $H=\frac12$, is quite similar to the
proof above. Thus the desired result is obtained.
\end{proof}

Recall that if $X\sim \mathcal{N}(m_1,\sigma_1)$ and $Y\sim
\mathcal{N}(m_2,\sigma_2)$ are two independent random variables,
then $X/Y$ follows a   Cauchy-type distribution. For a motivation
and further references, we refer the reader to \cite{PTM}, as well
as \cite{marsaglia}. Notice also that if $N\sim\mathcal{N}(0,1)$ is
independent of $B^H$, then $N$ is independent of $Z_{\infty}$, since
$Z_\infty = \int_0^{\infty}e^{-\theta s}B^H_sds$ is a functional of
$B^H$.

\begin{theorem}
Assume $\frac12\leq H<1$, and let
$N_1\sim\mathcal{N}(0,\sigma_H^2)$, $N_2\sim\mathcal{N}(0,D)$, where
 the variance $\sigma_H^2$ and  the covariance matrix $D$ are   given in
Lemma \ref{lemma-cv-couple-law}. Suppose that $N_1$, $N_2$ and $B^H$
are independent. Then, as
 $n\rightarrow\infty$,
\begin{eqnarray}e^{\alpha n}(\hat{\alpha}_n-\alpha)\overset{Law}{\longrightarrow}\frac{2\alpha N_1}{A_{\infty}
+\alpha Z_{\infty}},\label{cv-alpha-law}
\end{eqnarray}
 \begin{eqnarray}\left(e^{\alpha n}(\hat{\alpha}_n-\alpha),n^{1-H}\left(\hat{\mu}_n-\mu
\right)\right)\overset{Law}{\longrightarrow}\left(\frac{2\alpha
N_1}{A_{\infty} +\alpha Z_{\infty}},
N_2\right).\label{cv-joint-alpha-mu-law}
\end{eqnarray}
Consequently, as $n\rightarrow\infty$,
\begin{eqnarray}n^{1-H}\left(\hat{\mu}_T-\mu
\right)\overset{Law}{\longrightarrow} N_2.\label{cv-mu-law}
\end{eqnarray}
\end{theorem}

\begin{proof}Before proceeding to the proof of
\eqref{cv-joint-alpha-mu-law} we need first to prove
\eqref{cv-alpha-law}.   From \eqref{estimator-alpha}
  we can write
\begin{eqnarray}
e^{\alpha n}(\hat{\alpha}_{n}-\alpha)&=& \frac{e^{-\alpha n}\int_0^n
e^{\theta s} dB^H_s}{A_{\infty} +\alpha Z_{\infty}} \times
\frac{\left(A_{\infty} +\alpha Z_{\infty}\right)\left(A_{\infty}
+\alpha Z_n \right)}{e^{-2\alpha n}n\gamma_n^{-1}} +\frac{e^{-\alpha
n}S_{n,H}}{e^{-2\alpha n}n\gamma_n^{-1}}\nonumber \\
&=:&a_n\times b_n + c_n.\label{decomp-alpha}
\end{eqnarray}
Lemma \ref{lemma-cv-couple-law}  yields , as $n\rightarrow\infty$,
\begin{eqnarray}a_n\overset{
Law}{\longrightarrow}\frac{N_1}{A_{\infty} +\alpha
Z_{\infty}},\label{cv-law-a}\end{eqnarray}
 whereas \eqref{ii_lemma1}, \eqref{vi_lemma2} and \eqref{iii_lemma3} imply
 that
 $b_n
\longrightarrow 2\alpha$ almost surely, as $n\rightarrow\infty$. On
the other hand, by \eqref{cv-S}, \eqref{vi_lemma2} and
\eqref{iii_lemma3}, we obtain that $c_n \longrightarrow 0$ almost
surely as
$n\rightarrow\infty$.\\
Combining all these facts together with (\ref{cv-law-couple-F-G})
and Slutsky's theorem, we deduce \eqref{cv-alpha-law}.
\\
According to \eqref{estimator-mu}, we can write
\begin{eqnarray*}n^{1-H}(\hat{\mu}_n-\mu)&=&e^{\alpha n}(\alpha-\hat{\alpha}_{n})\times n^{1-H}e^{-\alpha n}(\Lambda_{n,1}, \ldots,
\Lambda_{n,p})+\frac{G_n}{n^H}\\
&=:&d_n+\frac{G_n}{n^H}.
\end{eqnarray*}
This combined with \eqref{decomp-alpha} allow to write
 \begin{eqnarray*}\left(e^{\alpha n}(\hat{\alpha}_n-\alpha),n^{1-H}\left(\hat{\mu}_n-\mu
\right)\right)&=&\left(a_n\times b_n + c_n,d_n+\frac{G_n}{n^H}\right)\\
&=&\left(a_n\times b_n ,\frac{G_n}{n^H}\right)+(c_n,d_n)
\\
&=&\left(2\alpha a_n ,\frac{G_n}{n^H}\right)+(a_n
(b_n-2\alpha),0)+(c_n,d_n).
\end{eqnarray*}
It follows from Lemma \ref{lemma-cv-couple-law}  that , as
$n\rightarrow\infty$,
\[\left(2\alpha a_n ,\frac{G_n}{n^H}\right)\overset{ Law}{\longrightarrow}\left(\frac{2\alpha N_1}{A_{\infty} +\alpha Z_{\infty}},N_2\right).\]
Furthermore, combining \eqref{cv-law-a} and $b_n \longrightarrow
2\alpha$ almost surely, together with   Slutsky's theorem, we obtain
$(a_n (b_n-2\alpha),0)\longrightarrow0$ in probability as
$n\rightarrow\infty$. Also, by \eqref{vi_lemma2},
\eqref{cv-alpha-law} and Slutsky's theorem, we have $d_n
\longrightarrow 0$  in probability as $n\rightarrow\infty$. Then,
since $c_n \longrightarrow 0$ almost surely, we get
$(c_n,d_n)\longrightarrow0$ in probability as $n\rightarrow\infty$.
Therefore, applying Slutsky's theorem, the convergence in law
\eqref{cv-joint-alpha-mu-law} is obtained.
\end{proof}

\section{Appendix}

As a direct consequence of the Borel-Cantelli Lemma, we found  the
following result, see (\cite{KN}, Lemma 2.1). Which ensures the
almost sure convergence from $L^p$ convergence.
\begin{lemma}\label{Borel-lemma}
Let $\gamma>0$. Let $(Z_{n})_{n\in\N}$ be a sequence of random
variables. If for every $p\geq 1$ there exists a constant $c_p>0$
such that for all $n\in\N$,
\begin{equation*}
(E\vert Z_n \vert^p)^{1/p}\leq c_p.n^{-\gamma},
\end{equation*}
then for all $\epsilon>0$ there exists a random variable
$\beta_{\epsilon}$ such that
\begin{equation*}
\vert Z_n\vert\leq
\beta_{\epsilon}.n^{-\gamma+\epsilon},\quad\text{almost\,surely},
\end{equation*}
for all $n\in\N$. Moreover, $E\vert \beta_{\epsilon}\vert^p<\infty$
for all $p\geq 1$.
\end{lemma}

In what follows, we briefly recall some basic elements of Young
integral (see \cite{Young}),  which are helpful for some of the
arguments we use in this paper. For any $\alpha\in [0,1]$, we denote
by $\mathscr{H}^\alpha([0,T])$ the set of $\alpha$-H\"older
continuous functions, that is, the set of functions $f:[0,T]\to\R$
such that
\[
|f|_\alpha := \sup_{0\leq s<t\leq
T}\frac{|f(t)-f(s)|}{(t-s)^{\alpha}}<\infty.
\]
We also set $|f|_\infty=\sup_{t\in[0,T]}|f(t)|$, and we equip
$\mathscr{H}^\alpha([0,T])$ with the norm \[ \|f\|_\alpha :=
|f|_\alpha + |f|_\infty.\] Let $f\in\mathscr{H}^\alpha([0,T])$, and
consider the operator $T_f:\mathcal{C}^1([0,T])
\to\mathcal{C}^0([0,T])$ defined as
\[
T_f(g)(t)=\int_0^t f(u)g'(u)du, \quad t\in[0,T].
\]
It can be shown (see, e.g., \cite[Section 3.1]{nourdin}) that, for
any $\beta\in(1-\alpha,1)$, there exists a constant
$C_{\alpha,\beta,T}>0$ depending only on $\alpha$, $\beta$ and $T$
such that, for any $g\in\mathcal{C}^1([0,T])$,
\[
\left\|\int_0^\cdot f(u)g'(u)du\right\|_\beta \leq
C_{\alpha,\beta,T} \|f\|_\alpha \|g\|_\beta.
\]
We deduce that, for any $\alpha\in (0,1)$, any
$f\in\mathscr{H}^\alpha([0,T])$ and any $\beta\in(1-\alpha,1)$, the
linear operator
$T_f:\mathcal{C}^1([0,T])\subset\mathscr{H}^\beta([0,T])\to
\mathscr{H}^\beta([0,T])$, defined as $T_f(g)=\int_0^\cdot
f(u)g'(u)du$, is continuous with respect to the norm
$\|\cdot\|_\beta$. By density, it extends (in an unique way) to an
operator defined on $\mathscr{H}^\beta$. As consequence, if
$f\in\mathscr{H}^\alpha([0,T])$, if $g\in\mathscr{H}^\beta([0,T])$
and if $\alpha+\beta>1$, then the (so-called) Young integral
$\int_0^\cdot f(u)dg(u)$ is (well) defined as being $T_f(g)$.

The Young integral obeys the following formula. Let
$f\in\mathscr{H}^\alpha([0,T])$ with $\alpha\in(0,1)$, and
$g\in\mathscr{H}^\beta([0,T])$ with $\beta\in(0,1)$. If
$\alpha+\beta>1$, then $\int_0^. g_udf_u$ and $\int_0^. f_u dg_u$
are well-defined as Young integrals, and for all $t\in[0,T]$,
\begin{eqnarray}\label{IBP}
f_tg_t=f_0g_0+\int_0^t g_udf_u+\int_0^t f_u dg_u.
\end{eqnarray}

\end{document}